\documentclass{amsart}
\usepackage{amssymb}
\usepackage{mathrsfs}
\usepackage{amscd} 
 \usepackage[normalem]{ulem}
 \usepackage{color}
 
\usepackage{hyperref}
\usepackage{graphicx}
\usepackage[curve]{xypic}
\usepackage[usenames,dvipsnames,x11names]{xcolor}
\usepackage{tikz}
\usetikzlibrary{arrows,snakes,shapes,shadows}
\usepackage{verbatim}

\newcommand{\bla}{weak RL} 
\newcommand{\Bla}{Weak RL}

\newbox\mybox
\def\overtag#1#2#3{\setbox\mybox\hbox{$#1$}\hbox to
  0pt{\vbox to 0pt{\vglue-#3\vglue-\ht\mybox\hbox to \wd\mybox
      {\hss$\ss#2$\hss}\vss}\hss}\box\mybox}
\def\undertag#1#2#3{\setbox\mybox\hbox{$#1$}\hbox to 0pt{\vbox to
    0pt{\vglue#3\vglue\ht\mybox\hbox to \wd\mybox
      {\hss$\ss#2$\hss}\vss}\hss}\box\mybox}
\def\lefttag#1#2#3{\hbox to 0pt{\vbox to 0pt{\vglue -6pt\hbox to
      0pt{\hss$\ss#2$\hskip#3}\vss}}#1}
\def\righttag#1#2#3{\hbox to 0pt{\vbox to 0pt{\vglue -6pt\hbox to
      0pt{\hskip#3$\ss#2$\hss}\vss}}#1}
\let\ss\scriptstyle
\def\splicediag#1#2{\xymatrix@R=#1pt@C=#2pt@M=0pt@W=0pt@H=0pt}
\def\Dot{\lower.2pc\hbox to 2pt{\hss$\bullet$\hss}}
\def\Circ{\lower.2pc\hbox to 2pt{\hss$\circ$\hss}}
\def\Vdots{\raise5pt\hbox{$\vdots$}}
\newcommand\lineto{\ar@{-}}
\newcommand\dashto{\ar@{--}}
\newcommand\dotto{\ar@{.}}

\let\cal\mathcal
\renewcommand{\setminus}{\smallsetminus}

\newcommand\C{{\mathbb C}}

\renewcommand\L{{$\cal L$}}

\newtheorem{theorem}{Theorem}[section]
\newtheorem{proposition}[theorem]{Proposition}
\newtheorem*{theorem*}{Theorem}
\newtheorem*{proposition*}{Proposition}

\newtheorem{lemma}[theorem]{Lemma}

\theoremstyle{definition}

\newtheorem{example}[theorem]{Example}

\newtheorem*{remark*}{Remark}
\newtheorem{definition}[theorem]{Definition}

\newcommand{\proj}{\operatorname{proj}}
\renewcommand{\int}{\operatorname{int}}

\begin{document}

\title[Lipschitz geometry does not determine embedded topological type]
{Lipschitz geometry does not determine embedded topological type} 
\author{Walter D Neumann} \address{Department of Mathematics,
  Barnard College, Columbia University, 2009 Broadway MC4424, New
  York, NY 10027, USA} \email{neumann@math.columbia.edu}

\author{Anne Pichon} \address{Aix Marseille Universit\'e, CNRS,
  Centrale Marseille, I2M, UMR 7373, 13453 Marseille, FRANCE}
\email{anne.pichon@univ-amu.fr} 

\dedicatory{Dedicated to Jos\'e Seade for a great  occasion. Happy birthday, Pepe!}
\begin{abstract} 
 We investigate the relationships between the Lipschitz outer geometry and the embedded topological type of a hypersurface germ in $(\C^n,0)$.
  It is well known that the Lipschitz outer geometry of a complex
  plane curve germ determines and is determined by its embedded
  topological type.  We
  prove that this does not remain true in higher dimensions. Namely, we
  give two normal hypersurface germs $(X_1,0)$ and $(X_2,0)$ in
  $(\C^3,0)$ having the same outer Lipschitz geometry and different
  embedded topological types. Our pair consist of two superisolated
  singularities whose tangent cones form an Alexander-Zariski pair
  having only cusp-singularities.  Our result is based on a    description of the Lipschitz outer geometry of a superisolated singularity. We
  also prove that the Lipschitz inner geometry of a superisolated
  singularity is completely determined by  its  (non embedded) topological type, or equivalently by  the combinatorial type of
  its tangent cone. \end{abstract}

\maketitle

\section{Introduction}

A complex germ $(X,0)$ has two natural metrics up to bilipschitz
equivalence, the  \emph{outer metric} given by embedding $(X,0)$
in some $(\C^n,0)$ and taking distance in $\C^n$ and the \emph{inner
metric} given by shortest distance along paths in $X$.

In this paper we investigate the relationships between the Lipschitz outer geometry and the embedded topological type of a hypersurface germ in $(\C^n,0)$.

  It is well known that the Lipschitz outer geometry of a complex
  plane curve germ determines and is determined by its embedded
  topological type (\cite{PT}, see also \cite{fernandes} and \cite[Theorem 1.1.]{NP}).  We
  prove that this does not remain true in higher dimensions:
\begin{theorem}\label{th:main}
There exist two  hypersurface germs in $(\C^3,0)$ having same
Lipschitz outer geometry and distinct embedded topological type.
\end{theorem}
It is worth noting that for families of isolated hypersurfaces in
$\C^3$, the constancy of Lipschitz outer geometry implies constancy of embedded
topological type. Indeed, Varchenko proved in \cite{varchenko} that a
Zariski equisingular family of hypersurfaces in any dimension has
constant embedded topological type and it is proved in \cite{NP2} that
for a family of hypersurface singularities $(X_t,0)\subset(\C^3,0)$,
Zariski equisingularity is equivalent to constant Lipschitz outer
geometry. 

It should also be noted that the converse question, which consists of
examining which part of the outer Lipschitz geometry of a
hypersurface can be recovered from its embedded topological type seems
difficult. In particular the outer geometry of a normal complex
surface singularity determines its multiplicity (\cite[Theorem 1.2 (2)]{NP2}) so this
question somehow contains the Zariski multiplicity question.

In order to prove Theorem \ref{th:main} we construct two germs of
hypersurfaces in $(\C^3,0)$ having the same Lipschitz  outer geometry
and different embedded topological types. They consist of a pair of
superisolated singularities whose tangent cones form an
Alexander-Zariski pair of projective plane curves.

 A surface singularity $(X,0)$ is {\it superisolated} (SIS for short)
 if it is given by an equation 
$$ f_{d} (x,y,z)+f_{d+1} (x,y,z)+f_{d+2}(x,y,z)+\dots=0, $$
where $d \geq 2$, $f_k$ is a homogeneous polynomial of degree $k$
and the projective curve $\{f_{d+1}=0 \} \subset \mathbb P^2$ 
contains no singular point of the projective curve $C= \{[x:y:z]:
f_{d} (x,y,z)=0\}$.  In particular, the projectivized tangent cone $C$
of $(X,0)$ is reduced.  In the sequel we will just consider
SISs with equations 
$$ f_{d} (x,y,z)+f_{d+1} (x,y,z)=0\,. $$

\begin{definition}[Combinatorial type of a projective plane
  curve]\label{def:combinatorial type}
The \emph{combinatorial type} of a reduced projective plane curve  $C \subset \mathbb P^2$ is the homeomorphism type of a tubular neighborhood of it in $\mathbb
P^2$ (see, e.g., \cite[Remark 3]{ACT}; a more combinatorial version is also given there, which we describe in Section \ref{sec:superisolated}).
\end{definition}

It is well known  that the combinatorial type
of the projectivized tangent cone of a SIS $(X,0)$ determines the  topology of $(X,0)$. In fact, we will show:
 \begin{theorem} \label{thm:inner} (i). The Lipschitz inner geometry of a
  SIS determines and is determined by  the combinatorial type of its projectivized tangent cone.

  (ii).  There exist SISs with the same combinatorial types of their projectivized tangent cones but different Lipschitz outer geometry.
\end{theorem}
\smallskip\noindent{\bf Acknowledgments.} We are grateful to
H\'el\`ene Maugendre for fruitful conversations and for communicating
to us the equations of the tangent cones for the examples in
the proof of Theorem \ref{thm:inner} (ii). Walter Neumann was supported by NSF grant
DMS-1206760. Anne Pichon was supported by ANR-12-JS01-0002-01 SUSI. We are also grateful for the
hospitality and support of the following institutions: Columbia University, Institut de Math\'ematiques de Marseille, Aix Marseille
Universit\'e and CIRM Luminy, Marseille.

\section{Proof of Theorem 1.1}

The proof of Theorem \ref{th:main} will need Lemma \ref{le:RL implies BLA} and Proposition \ref{thm:outer} below, which will be proved in section \ref{sec:outer geometry}. First a definition:

\begin{definition} We say that two germs $(C_1,0)$ and $(C_2,0)$ of
  reduced irreducible plane curves are
  \emph{\bla-equivalent} if for $i=1,2$ there are holomorphic maps $h_i\colon (\C^2,0)\to (\C,0)$ with $(h_i^{-1}(0),0)=(C_i,0)$,  a homeomorphism $\psi \colon
  (\C^2,0) \to (\C^2,0)$, a constant $K \geq 1$ and a neighborhood
  $\cal U$
  of the origin in $\C^2$ such that for all $a,a' \in \cal U$.
\begin{align*}
\frac{1}{K} ||  h_2(\psi(a)) (1,\psi(a)) -  h_2(\psi(a'))&(1,\psi(a')) ||_{\C^3} \leq   || h_1(a)(1,a) - h_1(a')(1,a') ||_{\C^3} \hskip1,5cm
\\ & \leq  K  || h_2(\psi(a)) (1,\psi(a)) -  h_2(\psi(a')) (1,\psi(a')) ||_{\C^3}
\end{align*}
\end{definition} 
\begin{lemma}\label{le:RL implies BLA}
 \Bla-equivalence of reduced irreducible plane curve 
germs $(C_1,0)$ and $(C_2,0)$ does 
 not depend on the choice of their  defining functions $h_1$ and $h_2$. Moreover, it is implied by
 analytic equivalence of  $(C_1,0)$ and $(C_2,0)$ in the sense of Zariski \cite{Z1}
 (also called RL-equivalence or $\cal A$-equivalence).
\end{lemma}

\begin{proposition} \label{thm:outer} Let $(X,0)$ be a SIS with
  equation $f_d + f_{d+1} =0$.  The Lipschitz outer geometry of
  $(X,0)$ is determined by the combinatorial type of its projectivized
  tangent cone and by the \bla-equivalence classes of corresponding
  singularities of the projectivized tangent cones.
  \end{proposition}
 
\begin{proof}[Proof of Theorem \ref{th:main}] 
  Recall that a Zariski pair is a pair of projective curves $C_1,C_2
  \subset \mathbb P^2$ with the same combinatorial type but such that
  $(\mathbb P^2,C_1)$ is not homeomorphic to $(\mathbb P^2,C_2)$. The
  first example was discovered by Zariski: a pair of sextic curves
  $C_1$ and $C_2$, each with six cusps, distinguished by the fact that
  $C_1$ has the cusps lying on a quadric and $C_2$ does not. He
  constructed those of type $C_1$ in \cite{Z1} and conjectured type
  $C_2$, confirming their existence eight years later in \cite{Z2}. He
  distinguished their embedded topology by the fundamental groups of their 
  complements, but they can also be
  distinguished by their Alexander polynomials (Libgober
  \cite{libgober}) so they are called \emph{Alexander-Zariski pairs}.

  Let $(X_1,0)$ and $(X_2,0)$ be two SISs whose
  tangent cones are sextics of types $C_1$ and $C_2$ as above.
According to \cite{Z}, the analytic type
  of a cusp is uniquely determined, so its \bla-equivalence class is
  determined (Lemma \ref{le:RL implies BLA}). Then by Proposition \ref{thm:outer}, $(X_1,0)$ and $(X_2,0)$
  are outer Lipschitz equivalent.

  On the other hand, Artal showed that $(X_1,0)$ and $(X_2,0)$ do not
  have the same embedded topological type. In fact, he shows
  (\cite[Theorem 1.6 (ii)]{A}) that a Zariski pair is distinguished by
  its Alexander polynomials if and only if the corresponding
  SISs are distinguished by the Jordan block
  decompositions of their homological monodromies.
\end{proof}

\section{The inner geometry of a superisolated singularity}
\label{sec:superisolated}
 
We first recall how the topological type of a SIS is determined by the combinatorial type of its projectivized tangent cone.
We refer to 
\cite{ALM} for details.

A SIS $(X,0)\subset (\C^3,0)$ is resolved by
blowing up the origin of $(\C^3,0)$. The exceptional divisor of this
resolution of $(X,0)$ is the projectivized tangent cone $C$ of $(X,0)$
and one obtains the minimal good resolution by blowing up the
singularities of $C$ which are not ordinary double points until one
obtains a normal crossing divisor $C'$. Let $\Gamma$ be the dual graph
of this resolution. Following \cite{BNP} we say \L-curve for a
component of $C'$ which is a component of $C$ and \L-node any vertex
of $\Gamma$ representing an \L-curve. 

One can also resolve the singularities of $C$ as a projective
  plane curve to obtain the same graph $\Gamma$ except that the
  self-intersection numbers of the \L-curves are different (in the
  example below the self-intersection number $-9$ becomes $+3$). The
  graph $\Gamma$ with these data is equivalent to the combinatorial
  type of $C$.

\begin{example} \label{ex:resolution} Consider the SIS
  $(X,0)\subset (\C^3,0)$ given by $F(x,y,z)=y^3+xz^2-x^4=0$. Blowing
  up the origin of $\C^3$ resolves the singularity: using the chart
  $(x,v,w)\mapsto (x,y,z)= (x,xv,xw)$, the equation of the resolved
  $X^*$ is $v^3+w^2-x=0$ and the exceptional curve has a cusp
  singularity $x=v^3+w^2=0$. Blowing up further leads to the following dual
  graph $\Gamma$, the black vertex 
  being the \L-node.
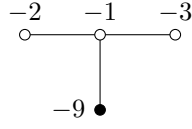
\begin{center}
\begin{tikzpicture}
\draw[thin] (0,0)--(2,0);
\draw[thin] (1,0)--(1,-1); 
\draw[fill=white] (0,0) circle(2pt);
\draw[fill=white] (1,0) circle(2pt);
\draw[fill=white] (2,0) circle(2pt);
\draw[fill=black] (1,-1) circle(2pt);
\node(a)at(0,.3){$-2$};
\node(a)at(1,.3){$-1$};
\node(a)at(2,.3){$-3$};
\node(a)at(.58,-1){$-9$};
\end{tikzpicture}
\end{center}

The self-intersection $-9$ of the \L-curve  is computed as follows.  Let $E_1,\ldots,E_4$ be the components of the exceptional divisor indexed so that $E_1$ is the \L-curve and $E_2$, $E_3$ and $E_4$ correspond to the string of non \L-nodes indexed from left to right  on the graph. Since the tangent cone is reduced with degree $3$, the strict transform $l_1^*$ of a generic linear form $l_1 \colon (X,0) \to (\Bbb C,0)$ consists of three smooth curves transverse to $E_1$. The total transform $l_1$ is given by the divisor:
$$ (l_1) =  E_1 + 3 E_2 + 6 E_3 + 2E_4 + l_1^*\,.$$ 
Since $(l_1)$ is a principal divisor, we have $(l_1).E_1=0$,
which leads to $E_1.E_1 = -9$. 
\end{example}

\begin{proof}[Proof of Theorem \ref{thm:inner} (i)]
  Let $(X,0) \subset (\mathbb C^3,0)$ be a SIS with equation
  $f_d  + f_{d+1} = 0$.  We set $f=f_d$ and $g=-f_{d+1}$. 
    
  Let $\ell \colon \C^3 \to \C^2$ be a generic linear projection for
  $(X,0)$, let $\Pi$ be the polar curve of the restriction
  $\ell \mid_{X}$ and $\Delta = \ell(\Pi)$ its discriminant curve.
 
  Let $e$ be the blow-up of the origin of $\C^3$ and let $p$ be a singular point
  of $e^{-1}(0)\cap X^*$. Without loss of generality, we
  can assume $\ell = (x,y)$.  We can also choose our coordinates so
  that $p = (1,0,0) $ in the chart $(x,v,w)$ given by $(x,v,w)\mapsto
  (x,y,z)= (x,xv,xw)$ in the blow-up $e$ (so $p$ corresponds to the
  $x$-axis in the tangent cone of $X$). Then $X^*$ has equation
$$ f(1,v,w) - x g(1,v,w) = 0$$
and $g(1,v,w)$ is a unit at $p$ since
$\{g=0\} \cap Sing (f=0)=\emptyset$ in $\Bbb P^2$.

Let $e_0 \colon Y \to \C^2$ be the blow-up of the origin of $\C^2$. We
consider $e_0$ in the chart $(x,v) \mapsto (x,y)=(x,xv)$, we set
$q=(1,0) \in Y$ in this chart, and we denote by
$\tilde{\ell} \colon (X^*,p) \to (Y,q)$ the projection
$(x,v,w) \mapsto (x,v)$. So we have the commutative diagram:
 $$
\xymatrix{
  (X^*,p)\ar@{->}[r]^e\ar@{->}[d]^{\tilde \ell}&(X,0)\ar@{->}[d]^\ell \\
 (Y,q)\ar@{->}[r]^{e_0}&  (\C^2,0)\hbox to 0pt{\,.\hss} 
}
$$
Now $\Pi = X \cap \{f_z - g_z = 0\}$, so the strict transform   $\Pi^*$ of $\Pi$ by $e$ has equations:
$$f_w(1,v,w) - x g_w(1,v,w) =0 \quad\text{and}\quad f(1,v,w) -x g(1,v,w) =0\,,$$
which are also the equations of the polar curve of the projection
$\tilde{\ell} \colon (X^*,p) \to (Y,q)$.

Since $g(1,v,w) \in \C\{v,w\}$ is a unit at $p$, the quotient $h(v,w):= \frac{f(1,v,w)}{g(1,v,w)}$ defines a  holomorphic function germ  $h \colon (\C^2_{(v,w)},0) \to (\C,0)$. In terms of $h(v,w)$ the above equations for $(\Pi^*,p)$ can be written:
$$h_w(v,w)  =0 \quad\text{and}\quad h(v,w) - x =0\,.$$

Consider the isomorphism $ proj \colon (X^*,p) \to (\C^2,0)$ which is the
restriction of the linear projection $(x,v,w)\mapsto (v,w)$. Then
$\Pi^*$ is the inverse image by $proj$ of the polar curve $\Pi'$ of
the morphism $\ell' \colon (\C_{(v,w)}^2,0) \to (\C_{(x,v)}^2,0)$
defined by $(v,w) \mapsto (h(v,w),v)$, i.e., the relative  polar curve of the
map germ $(v,w) \mapsto h(v,w)$ for the generic projection $(v,w)
\mapsto v$.

We set
$\Delta' = \ell'(\Pi')$ and $q=(1,0)$ in $\C^2_{(x,v)}$.  We then have a  commutative diagram: 
 $$
\xymatrix{
 (\C^2,\Pi',0)\ar@{<-}[r]^{proj}\ar@{->}[dr]^{\ell'} & (X^*,\Pi^*,p)\ar@{->}[r]^e\ar@{->}[d]^{\tilde \ell}&(X,\Pi,0)\ar@{->}[d]^\ell \\
& (Y,\Delta',q)\ar@{->}[r]^{e_0}&  (\C^2,\Delta, 0) 
}
$$
 
Let $(\Pi_0,0)$ be the part of $(\Pi,0)$ which is tangent to the
$x-$axis (i.e., it corresponds to $p\in e^{-1}(0)$ in our chosen
coordinates) and let $(\Delta_0,0)$ be its image by $\ell$. Let $V$ be
a cone around the $x$-axis in $(\C^3,0)$. As in \cite{BNP}, consider a
carrousel decomposition of $(\ell(V),0)$ with respect to the curve
germ $(\Delta_0, 0)$ such that the $\Delta$-wedges around $\Delta_0$
are D-pieces. We then consider the geometric decomposition of $(V,0)$
into A-, B- and D-pieces obtained by lifting by $\ell$ this
decomposition. Lifting the carrousel decomposition of $\ell(V)$ by
$e_0$ we get a carrousel decomposition of $(Y,q)$ with respect to
$\Delta'$. On the other hand the lifting by $e$ of the geometric
decomposition of $V$ is a geometric decomposition of $(X^*,p)$ which
coincides with the lifting by $\tilde \ell$ of the carrousel
decomposition of $(Y,q)$ just defined.

By the L\^e Swing Lemma \cite[Lemma 2.4.7]{LMW}, the union of pieces
beyond the first Puiseux exponents of the branches of $\Delta'$ at $q$
lift to pieces in $X^*$ which have trivial topology, i.e., their links
are solid tori.  Therefore these are absorbed by the amalgamation
process consisting of amalgamating iteratively any D-piece which is
not a conical piece with the neighbor piece using \cite[Lemma 13.1]{BNP}.

Moreover, since $\Delta'$ is the strict transform of $\Delta$ by
$e_0$, the rate of each piece of the obtained decomposition of $X^*$
equals $q+1$, where $q$ is the first Puiseux exponent of a branch of
$\Delta'$.  Let $\Gamma_p$ be the minimal resolution graph of the
curve $h=0$ at $p$. Let us call a {\it node} of $\Gamma_p$ any vertex
having at least three incident edges including the arrows representing
the components of $h$ and the root vertex of $\Gamma_p$ if $h=0$ has
more than one line in its tangent cone.  According to
\cite[Th\'eor\`eme C]{LMW}, the rate $q$ equals the polar quotient
$$ \frac{m_{E_i}(l)}{m_{E_i}(h)}$$
where $v_i$ is the corresponding node in $\Gamma_p$ and where $l
\colon (\C_{v,w}^2,p) \to (\C,0)$ is a generic linear form at $p$.
  
Now, set $\tilde{f}(v,w)=f(1,v,w)$.  Since $g(1,v,w)$ is a unit at
$p$, the curves $h=0$ and $\tilde{f}=0$ coincide, so $m_{E_i}(h) =
m_{E_i}(\tilde{f})$.  Since the strict transform of $\tilde{f}$
coincides with the germ of $\cal L$-curves at $p$, $\Gamma_p$ is a
connected component of $\Gamma$ minus its \L-nodes with free edges
replaced by arrows.  Therefore the rates
$\frac{m_{E_i}(l)}{m_{E_i}(\tilde f)}$, and then the inner rate of
$(X, 0)$ are computed from $\Gamma$.
\end{proof}

\begin{example} \label{ex:inner geometry} Consider again the
  SIS $(X,0)$ of Example \ref{ex:resolution} with equation $
  xz^2+y^3-x^4=0$. Its projectivized tangent cone $xz^2+y^3=0$ has a unique singular
  point, and the corresponding graph $\Gamma_p$ is the resolution
  graph of the cusp $w^2+v^3=0$, i.e., the graph $\Gamma$ of Example
  \ref{ex:resolution} with the \L-node replaced by an arrow. The
  multiplicity of $\tilde f$ along the curve $E_3$ corresponding to
  the node of $\Gamma_p$ equals $6$ while that of a generic linear
  form $(v,w)\mapsto l(v,w)$ equals $2$. We then obtain the polar
  quotient $\frac{m_{E_3}(l)}{m_{E_3}(\tilde f)} = 1/3$, which gives
  inner rate $1/3 + 1 = 4/3$.
 
  The Lipschitz inner geometry is then completely described (see
  \cite[Section 15]{BNP}) by the graph $\Gamma$ completed by labeling
  its nodes by the inner rates of the corresponding geometric pieces:
 
 
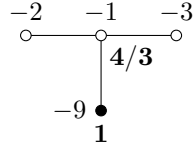
\begin{center}
\begin{tikzpicture}
\draw[thin] (0,0)--(2,0);
\draw[thin] (1,0)--(1,-1); 
\draw[fill=white] (0,0) circle(2pt);
\draw[fill=white] (1,0) circle(2pt);
\draw[fill=white] (2,0) circle(2pt);
\draw[fill=black] (1,-1) circle(2pt);
\node(a)at(0,.3){$-2$};
\node(a)at(1,.3){$-1$};
\node(a)at(2,.3){$-3$};
\node(a)at(.58,-1){$-9$};

\node(a)at(1,-1.3){$\mathbf 1$};

\node(a)at(1.4,-0.3){$\mathbf{4/3}$};

\end{tikzpicture}
\end{center}

\end{example}
  
\begin{example} \label{ex:inner geometry2} Consider the SIS $(X,0)$ with equation $ (zx^2+y^3)(x^3+zy^2)+z^7= 0$,
  that we already considered in \cite[Example 15.2]{BNP} and in
  \cite{NP2}. The tangent cone consists of two unicuspidal curves $C$
  and $C'$ with $6$ intersecting points $p_1, \ldots p_6$, the germ
  $(C \cup C',p_1)$ consisting of two transversal cusps, and the
  remaining 5 points being ordinary double points of $C \cup
  C'$.
   
  For each $i=1,\ldots,6$, the tangent cone of $(C \cup C',p_i)$ has
  two tangent lines and the quotient $m_{E_{v_0}}(l) / m_{E_{v_0}}(\tilde f)$
  at the root vertex $v_0$ of $\Gamma_{p_i}$ is then a polar quotient
  in the sense of \cite{LMW}. The root vertex $v_0$ has valency $2$
  and it corresponds to a special annular piece in the sense of
  \cite{BNP}, with inner rate $m_{E_{v_0}}(l) / m_{E_{v_0}}(\tilde f)
  +1$. For $p_2,\ldots,p_6$, we obtain inner rate $1/2 + 1 = 3/2$ for
  that special annular piece and for $p_1$, we obtain $1/4 + 1 =
  5/4$. The inner rates at the two other nodes of $\Gamma_{p_1}$ both
  equal $2/10 + 1 = 6/5$. We have thus recovered the inner geometry:
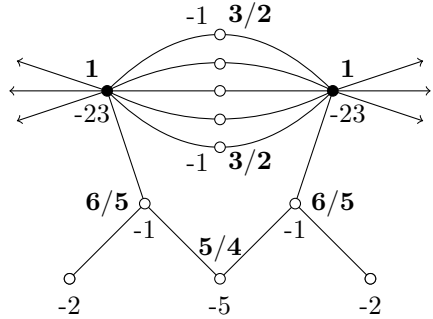
\begin{center}
\begin{tikzpicture}
  \draw[] (-2,0)circle(2pt);
  \draw[thin ](-2,0)--(-1,1);

  \draw[thin ](0:0)--(-1,1);
 
     \draw[thin ](0:0)--(1,1);
     
        \draw[thin ](1,1)--(2,0);
            \draw[] (2,0)circle(2pt);
   \draw[thin ](1,1)--(1.5,2.5);
    \draw[thin ](-1,1)--(-1.5,2.5);
    \draw[ fill] (1.5,2.5)circle(2pt);
     \draw[ fill] (-1.5,2.5)circle(2pt);
     \draw[thin ](-1.5,2.5)--(1.5,2.5);
     
\draw[fill=white] (2,0)circle(2pt);
\draw[fill=white] (-2,0)circle(2pt);
      
\draw[thin] (-1.5,2.5)..controls (-0.5,3) and (0.5,3)..(1.5,2.5);
\draw[thin] (-1.5,2.5)..controls (-0.5,2) and (0.5,2)..(1.5,2.5);
\draw[thin] (-1.5,2.5)..controls (-0.5,3.5) and (0.5,3.5)..(1.5,2.5);
\draw[thin] (-1.5,2.5)..controls (-0.5,1.5) and (0.5,1.5)..(1.5,2.5);

         \draw[fill=white] (0,2.5)circle(2pt);
           \draw[fill=white] (0,2.86)circle(2pt);
             \draw[fill=white] (0,2.12)circle(2pt);
               \draw[fill=white] (0,1.75)circle(2pt);
                 \draw[fill=white] (0,3.25)circle(2pt);
\draw[fill=white] (-1,1)circle(2pt);
\draw[fill=white] (1,1)circle(2pt);
\draw[fill=white] (0,0)circle(2pt);
     
\node(a)at(-2,-0.35){-2};
\node(a)at(0,-0.35){-5};
\node(a)at(2,-0.35){-2};
\node(a)at(-1,0.65){-1};
\node(a)at(1,0.65){-1};
\node(a)at(-0.3,3.5){-1};
\node(a)at(0.4,3.5){$\mathbf{3/2}$};
\node(a)at(-0.3,1.55){-1};
\node(a)at(0.4,1.55){$\mathbf{3/2}$};
\node(a)at(-1.7,2.8){$\mathbf{1}$};
\node(a)at(-1.7,2.2){-23};
\node(a)at(1.7,2.8){$\mathbf{1}$};

\node(a)at(1.7,2.2){-23};
\node(a)at(-1.5,1){$\mathbf{6/5}$};
\node(a)at(1.5,1){$\mathbf{6/5}$};
\node(a)at(0,0.4){$\mathbf{5/4}$};

 \draw[thin,>-stealth,->](1.5,2.5)--+(1.2,0.4);
       \draw[thin,>-stealth,->](1.5,2.5)--+(1.3,0);
         \draw[thin,>-stealth,->](1.5,2.5)--+(1.2,-0.4);
         
            \draw[thin,>-stealth,->](-1.5,2.5)--+(-1.2,0.4);
       \draw[thin,>-stealth,->](-1.5,2.5)--+(-1.3,0);
         \draw[thin,>-stealth,->](-1.5,2.5)--+(-1.2,-0.4);
\end{tikzpicture} 
\end{center} 
  This was also computed in \cite{BNP} with the help of Maple, in terms of
  the carrousel decomposition of the discriminant curve of a generic
  projection of $(X,0)$.
\end{example}

\begin{proof}[Proof of Theorem \ref{thm:inner} (ii)] 
 Consider the two SISs $(X_1,0)$ and $(X_2,0)$ with equations respectively:
 \begin{align*}
  X_1:& \qquad F_1(x,y,z)= (y^3-z^2x)(y^3 + z^2x) + (x+y+z)^7=0\\
 X_2:& \qquad F_2(x,y,z)=(y^3-z^2x)(y^3 + 2 z^2x) + (x+y+z)^7=0
 \end{align*}
 We will prove that they have same inner geometry and different outer geometries.
 
On one hand, the  projectivized tangent cones  of  $(X_1,0)$ and
$(X_2,0)$ have same  combinatorial type, so $(X_1,0)$ and $(X_2,0)$
have same Lipschitz inner geometry (Theorem \ref{thm:inner}). The
tangent cone consists of two unicuspidal components $C$ and $C'$ with
two intersection points: one, $p_1$,  at the cusps, with maximal
contact there, and one, $p_2$, at smooth points of $C$ and $C'$
intersecting with contact $3$ there.  The inner geometry is given by
the following graph. In particular, the    inner rates  at the two non
\L-nodes are computed from the corresponding  polar rates in the two
graphs $\Gamma_{p_1}$ and $\Gamma_{p_2}$. They both equal $1/6 + 1 = 7/6$. 

 
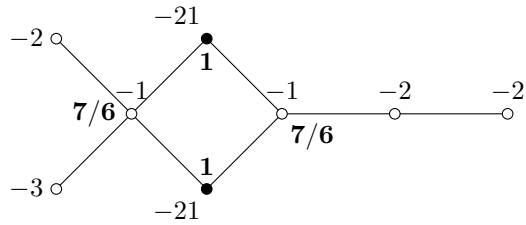
\begin{center}
\begin{tikzpicture}
\draw[thin] (0,0)--(3,0);
\draw[thin] (-1,-1)--(0,0); 
\draw[thin] (-1,1)--(0,0); 
\draw[thin] (-1,1)--(-2,0); 
\draw[thin] (-1,-1)--(-2,0); 
\draw[thin] (-3,-1)--(-2,0); 
\draw[thin] (-3,1)--(-2,0); 

\draw[fill=white] (0,0) circle(2pt);
\draw[fill=white] (1.5,0) circle(2pt);
\draw[fill=white] (3,0) circle(2pt);

\draw[fill=black] (-1,-1) circle(2pt);
\draw[fill=black] (-1,1) circle(2pt);
\draw[fill=white] (-2,0) circle(2pt);

\draw[fill=white] (-3,-1) circle(2pt);
\draw[fill=white] (-3,1) circle(2pt);

\node(a)at(0,.3){$-1$};
\node(a)at(1.5,.3){$-2$};
\node(a)at(3,.3){$-2$};
 
\node(a)at(0.4,-0.3){$\mathbf{7/6}$};
\node(a)at(-1.4,-1.3){$-21$};
\node(a)at(-1.4,1.3){$-21$};
\node(a)at(-1,.7){$\mathbf{1}$};
\node(a)at(-1,-.7){$\mathbf{1}$};
\node(a)at(-2.5,0){$\mathbf{7/6}$};
\node(a)at(-2,0.3){$-1$};
\node(a)at(-3.4,-1){$-3$};
\node(a)at(-3.4,1){$-2$};

\end{tikzpicture}
\end{center}

On the other hand, let us compute the multiplicities of the three functions $x, y$ and $z$ at each component of the exceptional locus. We obtain the following  triples $(m_{E_j}(x), m_{E_j}(y), m_{E_j}(z))$ for both $X_1$ and $X_2$:  

 
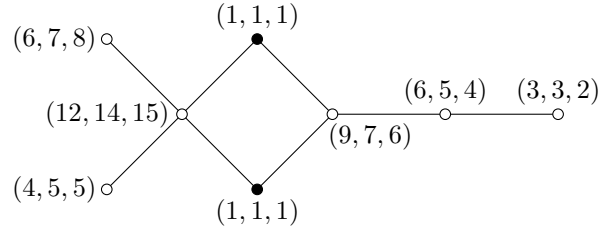
\begin{center}
\begin{tikzpicture}
\draw[thin] (0,0)--(3,0);
\draw[thin] (-1,-1)--(0,0); 
\draw[thin] (-1,1)--(0,0); 
\draw[thin] (-1,1)--(-2,0); 
\draw[thin] (-1,-1)--(-2,0); 
\draw[thin] (-3,-1)--(-2,0); 
\draw[thin] (-3,1)--(-2,0); 

\draw[fill=white] (0,0) circle(2pt);
\draw[fill=white] (1.5,0) circle(2pt);
\draw[fill=white] (3,0) circle(2pt);

\draw[fill=black] (-1,-1) circle(2pt);
\draw[fill=black] (-1,1) circle(2pt);
\draw[fill=white] (-2,0) circle(2pt);

\draw[fill=white] (-3,-1) circle(2pt);
\draw[fill=white] (-3,1) circle(2pt);

\node(a)at(3,.3){$(3,3,2)$};
\node(a)at(1.5,.3){$(6,5,4)$};
\node(a)at(0.5,-.3){$(9,7,6)$};

\node(a)at(-1,-1.3){$(1,1,1)$};
\node(a)at(-1,1.3){$(1,1,1)$};

\node(a)at(-3,0){$(12,14,15)$};

\node(a)at(-3.7,-1){$(4,5,5)$};
\node(a)at(-3.7,1){$(6,7,8)$};

\end{tikzpicture}
\end{center}

We compute from this the partial derivatives $\frac{\partial F_i}{\partial x}$, $\frac{\partial F_i}{\partial y}$ and $\frac{\partial F_i}{\partial z}$ along the curves of the exceptional divisor. We obtain different  values for two multiplicities (in bold) for $(X_1,0)$ and  $(X_2,0)$, written in that order on the graph:

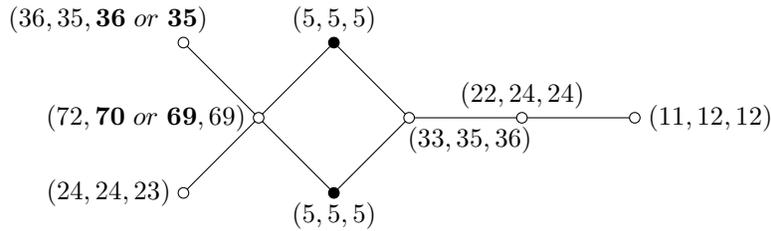
\begin{center}
\begin{tikzpicture}
\draw[thin] (0,0)--(3,0);
\draw[thin] (-1,-1)--(0,0); 
\draw[thin] (-1,1)--(0,0); 
\draw[thin] (-1,1)--(-2,0); 
\draw[thin] (-1,-1)--(-2,0); 
\draw[thin] (-3,-1)--(-2,0); 
\draw[thin] (-3,1)--(-2,0); 

\draw[fill=white] (0,0) circle(2pt);
\draw[fill=white] (1.5,0) circle(2pt);
\draw[fill=white] (3,0) circle(2pt);

\draw[fill=black] (-1,-1) circle(2pt);
\draw[fill=black] (-1,1) circle(2pt);
\draw[fill=white] (-2,0) circle(2pt);

\draw[fill=white] (-3,-1) circle(2pt);
\draw[fill=white] (-3,1) circle(2pt);

\node(a)at(4,0){$(11,12,12)$};
\node(a)at(1.5,.3){$(22,24,24)$};
\node(a)at(0.8,-.3){$(33,35,36)$};

\node(a)at(-1,-1.3){$(5,5,5)$};
\node(a)at(-1,1.3){$(5,5,5)$};

\node(a)at(-3.5,0){$\small (72,\mathbf{70} \ or\  \mathbf{69},69)$};

\node(a)at(-4,-1){$(24,24,23)$};
\node(a)at(-4,1.3){ $\small (36,35, \mathbf{36}\ or \ \mathbf{35} )$};
\end{tikzpicture}
\end{center}

We compute from this  the resolution graph of the family of polar curves $a \frac{\partial F_i}{\partial x} + b \frac{\partial F_i}{\partial y} +c\frac{\partial F_i}{\partial z} =0$. In the $X_1$ case one has to blow up once more to resolve a basepoint. We then get  the resolution graph of the polar curve of a generic plane projection of $(X_1,0)$ resp.\ $(X_2,0)$ (the arrows represent the strict transform,  the numbers in parentheses are the multiplicities of the function $a \frac{\partial F_i}{\partial x} + b \frac{\partial F_i}{\partial y} +c\frac{\partial F_i}{\partial z}$ for generic $a,b,c$ and the negative numbers are self-intersections): 
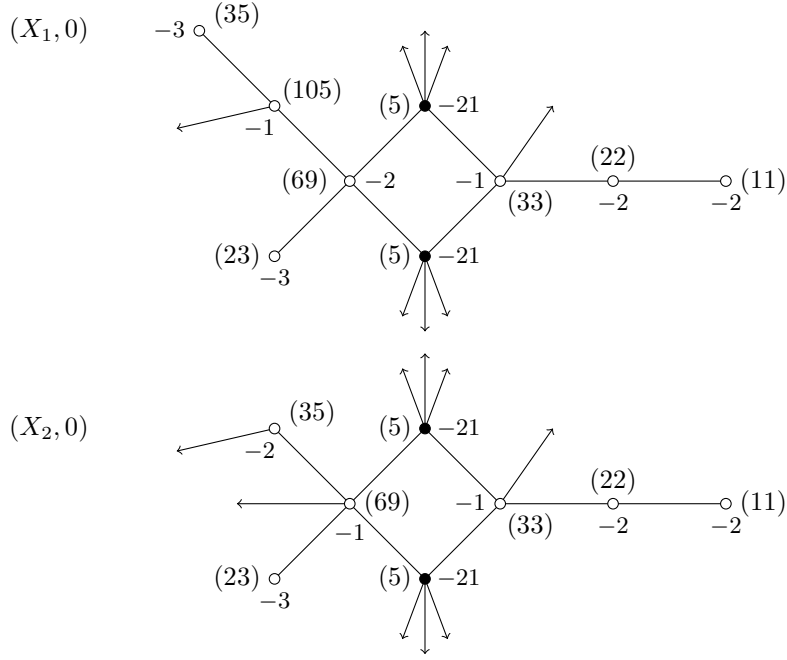
\begin{center}
\begin{tikzpicture}
\draw[thin] (0,0)--(3,0);
\draw[thin] (-1,-1)--(0,0); 
\draw[thin] (-1,1)--(0,0); 
\draw[thin] (-1,1)--(-2,0); 
\draw[thin] (-1,-1)--(-2,0); 
\draw[thin] (-3,-1)--(-2,0); 
\draw[thin] (-4,2)--(-2,0); 

 \draw[thin,>-stealth,->](0,0)--+(0.7,1);
 \draw[thin,>-stealth,->](-3,1)--+(-1.3,-0.3);
 
 \draw[thin,>-stealth,->](-1,1)--+(0,1);
  \draw[thin,>-stealth,->](-1,1)--+(0.3,0.8);
  \draw[thin,>-stealth,->](-1,1)--+(-0.3,0.8);
  
   \draw[thin,>-stealth,->](-1,-1)--+(0,-1);
  \draw[thin,>-stealth,->](-1,-1)--+(0.3,-0.8);
  \draw[thin,>-stealth,->](-1,-1)--+(-0.3,-0.8);

\draw[fill=white] (0,0) circle(2pt);
\draw[fill=white] (1.5,0) circle(2pt);
\draw[fill=white] (3,0) circle(2pt);

\draw[fill=black] (-1,-1) circle(2pt);
\draw[fill=black] (-1,1) circle(2pt);
\draw[fill=white] (-2,0) circle(2pt);

\draw[fill=white] (-3,-1) circle(2pt);
\draw[fill=white] (-3,1) circle(2pt);
\draw[fill=white] (-4,2) circle(2pt);

\node(a)at(3.5,0){$(11)$};
\node(a)at(1.5,.3){$(22)$};
\node(a)at(0.4,-.3){$(33)$};

\node(a)at(-1.4,-1){$(5)$};
\node(a)at(-1.4,1){$(5)$};

\node(a)at(-2.6,0){$\small (69)$};

\node(a)at(-3.5,-1){$(23)$};
\node(a)at(-3.5,2.2){ $\small (35)$};
\node(a)at(-2.5,1.2){ $\small (105)$};
 \node(a)at(-6,2){ $(X_1,0)$};

{\small 

\node(a)at(-0.4,0){$-1$};
\node(a)at(1.5,-.3){$-2$};
\node(a)at(3,-.3){$-2$};

\node(a)at(-0.55,-1){$-21$};
\node(a)at(-0.55,1){$-21$};

\node(a)at(-1.6,0){$-2$};
\node(a)at(-3.2,0.7){$-1$};
\node(a)at(-4.4,2){$-3$};

\node(a)at(-3,-1.3){$-3$};
}
\end{tikzpicture}
\end{center}
\begin{center}
\begin{tikzpicture}

\draw[thin] (0,0)--(3,0);
\draw[thin] (-1,-1)--(0,0); 
\draw[thin] (-1,1)--(0,0); 
\draw[thin] (-1,1)--(-2,0); 
\draw[thin] (-1,-1)--(-2,0); 
\draw[thin] (-3,-1)--(-2,0); 
\draw[thin] (-3,1)--(-2,0); 

 \draw[thin,>-stealth,->](0,0)--+(0.7,1);
 \draw[thin,>-stealth,->](-3,1)--+(-1.3,-0.3);
  \draw[thin,>-stealth,->](-2,0)--+(-1.5,0);

\draw[thin,>-stealth,->](-1,1)--+(0,1);
  \draw[thin,>-stealth,->](-1,1)--+(0.3,0.8);
  \draw[thin,>-stealth,->](-1,1)--+(-0.3,0.8);
  
   \draw[thin,>-stealth,->](-1,-1)--+(0,-1);
  \draw[thin,>-stealth,->](-1,-1)--+(0.3,-0.8);
  \draw[thin,>-stealth,->](-1,-1)--+(-0.3,-0.8);

\draw[fill=white] (0,0) circle(2pt);
\draw[fill=white] (1.5,0) circle(2pt);
\draw[fill=white] (3,0) circle(2pt);

\draw[fill=black] (-1,-1) circle(2pt);
\draw[fill=black] (-1,1) circle(2pt);
\draw[fill=white] (-2,0) circle(2pt);

\draw[fill=white] (-3,-1) circle(2pt);
\draw[fill=white] (-3,1) circle(2pt);

\node(a)at(3.5,0){$(11)$};
\node(a)at(1.5,.3){$(22)$};
\node(a)at(0.4,-.3){$(33)$};

\node(a)at(-1.4,-1){$(5)$};
\node(a)at(-1.4,1){$(5)$};

\node(a)at(-1.5,0){$\small (69)$};

\node(a)at(-3.5,-1){$(23)$};
\node(a)at(-2.5,1.2){ $\small (35)$};

{\small 

\node(a)at(-0.4,0){$-1$};
\node(a)at(1.5,-.3){$-2$};
\node(a)at(3,-.3){$-2$};

\node(a)at(-0.55,-1){$-21$};
\node(a)at(-0.55,1){$-21$};

\node(a)at(-2,-0.4){$-1$};
\node(a)at(-3.2,0.7){$-2$};

\node(a)at(-3,-1.3){$-3$};
}

 \node(a)at(-6,1){ $(X_2,0)$};
\end{tikzpicture}
\end{center}
The polar curves of $(X_1,0)$ and $(X_2,0)$  have different Lipschitz geometry since they  don't even have the same number of components. Therefore, by  \cite[Theorem  1.2 (6)]{NP2}, $(X_1,0)$ and $(X_2,0)$ have different outer Lipschitz geometries. 
\end{proof}

\section{The outer geometry of a superisolated
  singularity}\label{sec:outer geometry} \begin{proof}[Proof of
    Lemma \ref{le:RL implies BLA}] We first re-formulate the
    definition of \bla-equivalence.  We will use coordinates $(v,w)$
    in $\C^2$ and $(x,y,z)$ in $\C^3$.  We have functions $h_1(v,w)$
    and $h_2(v,w)$ whose zero sets are the curves $(C_1,0)$ and
    $(C_2,0)$, a homeomorphism $\psi \colon (\C^2,0) \to (\C^2,0)$ of
    germs, a constant $K \geq 1$ and a neighborhood $\cal U$ of the origin
    in $\C^2$ such that for all $a,a' \in \cal U$.  \begin{align*}
      \frac{1}{K} || h_2(\psi(a)) (1,\psi(a)) -
      h_2(\psi(a'))&(1,\psi(a')) ||_{\C^3} \leq || h_1(a)(1,a) -
      h_1(a')(1,a') ||_{\C^3} \hskip1,5cm \\ & \leq K || h_2(\psi(a))
      (1,\psi(a)) - h_2(\psi(a')) (1,\psi(a')) ||_{\C^3} \end{align*}
    For $i=1,2$ we define $H_i\colon (\C^2,0) \to (\C^3,0)$ by $$
    H_i(v,w)=h_i(v,w)(1,v, w)$$ and denote by $(S_i,0)$ the image of
    $H_i$ in $(\C^3,0)$. Note that $H_i$ maps $(C_i,0)$ to $0$ and is
    otherwise injective. We can thus complete the maps $\psi$, $H_1$
    and $H_2$ to a commutative diagram
$$\xymatrix{(\C^2,0)\ar@{->}[r]^{H_1}\ar@{->}[d]^\psi
  &(S_1,0)\ar@{->}[d]^{\psi'\hbox to 0pt{\hbox to 3.5cm{}$(\star)$\hss}}\\
      (\C^2,0)\ar@{->}[r]^{H_2} &(S_2,0)\\
    }
$$
 and $\psi'$ is bijective. \Bla-equivalence is now the
    statement that $\psi'$ is bilipschitz for the outer geometry.  

Now write
    $h_1=Uh'_1$ and $H_1 =UH'_1$ where $U=U(v,w)\in \C\{v,w\}$
    is a unit. Then we obtain a commutative diagram
    $$\xymatrix{(\C^2,0)\ar@{->}[r]^{H'_1}\ar@{=}[d] &(S'_1,0)\ar@{->}[d]^{\eta}\\
      (\C^2,0)\ar@{->}[r]^{H_1} &(S_1,0)\\
    }$$ 
   where $\eta$ is $(x,y,z)\mapsto U(\frac yx,\frac
    zx)(x,y,z)$. The factor $U(\frac yx,\frac zx)=U(v,w)$ has the form 
    $\alpha_0+\sum_{i,j\ge 0}\alpha_{ih}v^iw^j$ with $\alpha_0\ne 0$
    so if the neighborhood $\cal U$ is small then the factor is close
    to  $\alpha_0$, so $\eta$ is bilipschitz. Thus $\psi'\circ \eta\colon
 (S'_1,0)\to (S_2,0)$ is bilipschitz, so we have shown that modifying $h_1$
    by a unit does not affect \bla-equivalence. The same holds
     for
    $h_2$, so \bla-equivalence does not depend on the  choice of
    defining functions for the curves $(C_1,0)$
    and $(C_2,0)$.

    It remains to show that analytic equivalence of $(C_1,0)$ and
    $(C_2,0)$ implies \bla-equivalence. Analytic equivalence means
    that there exists a biholomorphic germ $\psi\colon(\C^2,0)\to
    (\C^2,0)$ and a unit $U\in\C\{v,w\}$ such that $Uh_1=h_2\circ
    \psi$. We have already dealt with multiplication with a unit, so
    we will assume we have $h_1=h_2\circ \psi$. If $\psi$ is a linear
    change of coordinates, then we get a diagram as in $(\star)$
    above, with $\psi'$ given by the corresponding coordinate change
    in the $y,z$ coordinates of $\C^3$, so $\psi'$ is bilipschitz and
    we have \bla-equivalence. For general $\psi$ the same is true up to higher
    order in $v$ and $w$, so we still get \bla-equivalence.
  \end{proof}

\begin{proof}[Proof of Proposition \ref{thm:outer}] 
 Let $(X_1,0)$ and $(X_2,0)$ be two SISs with equations respectively 
$$f_1(x,y,z) - g_1(x,y,z) =0 \hbox{ and  }  f_2(x,y,z) - g_2(x,y,z) =0,$$
where for $i=1,2$, $f_i$ and $g_i$ are homogeneous polynomials of
degrees $d$ and $d+1$ respectively. We can assume that the projective
line $x=0$ does not contain any singular point of the projectivized
tangent cones $C_1= \{f_1=0\}$ and $C_2 = \{f_2=0\}$.  We assume also
that $C_1$ and $C_2$ have the same combinatorial types and that
corresponding singular points of $C_1$ and $C_2$ are \bla-equivalent.

Since the tangent cone of a SIS $(X,0)$ is
  reduced, the general hyperplane section of $(X,0)$ consists of
  smooth transversal lines. Therefore, adapting the arguments of
  \cite[Section 4]{NPP} by taking simply a line as test curve, we
  obtain that the inner and outer metrics are Lipschitz equivalent
  inside the conical part of $(X,0)$, i.e., outside cones around its
  exceptional lines.  So we just have to control outer distance
inside conical neighborhoods of the exceptional lines of $(X_1,0)$ and
$(X_2,0)$ whose projective points are corresponding singular points of
$C_1$ and $C_2$.

Let $p_1 \in Sing(C_1)$ and $p_2 \in Sing(C_2)$ be two singular points
in correspondence.  After modifying $(X_1,0)$ and $(X_2,0)$ by
analytic isomorphisms, we can assume that $p_i = (1,0,0)$ for
$i=1,2$. We use again the notations of the proof of Theorem
\ref{thm:inner}, and we work in the chart $(x,v,w) = (x,y/x,z/x)$ for
the blow-up $e$.
  
Set $h_i(v,w)=f_i(1,v,w)/ g_i (1,v,w)$. Then the germs $(X_i^*,p_i)$
have equations $h_i(v,w)+x=0$.
  
Since $C_1$ and $C_2$ are \bla-equivalent and $h_i=0$ is an equation of
$C_i$, there exists a local homeomorphism $\psi \colon
(\C^2_{(v,w)},0) \to (\C^2_{(v,w)},0)$, a constant $K \geq 1$ and a
neighborhood $U$ of the origin in $\C^2$ such that for all
$(v,w),(v',w') \in U$.
\begin{align*}
\frac{1}{K} ||  h_2(\psi(v,w)) (1,\psi(v,w)) -  h_2(\psi(v',w')) (1,\psi(v',w')) ||_{\C^3} \leq& \\
|| h_1(v,w)(1,v,w) - h_1(v',w')(1,v',w') &||_{\C^3} \leq&  (\ast)\\
  K  ||  h_2(\psi(v,w)) (1,\psi(v,w)) -  h_2(\psi(v',w')) (1,&\psi(v',w')) ||_{\C^3}
\end{align*}
         
Locally, 
$$X_1^* = \{ x=h_1(v,w)\}
\quad\text{and}\quad 
X_2^* = \{ x = h_2(\psi(v,w)) \}\,.$$
As in the proof of Theorem \ref{thm:inner} we consider
the isomorphisms $\proj_i\colon (X^*_i,p_i)\to (\C^2,0)$  for $i=1,2$, the
restrictions of the linear projections $(x,v,w)\mapsto (v,w)$.  The
composition $\proj_2^{-1}\circ\psi\circ\proj_1$ gives a local
homeomorphism $\psi' \colon (W_1,p_1) \to (W_2,p_2)$, where $W_i$ is
an open neighborhood of $p_i$ in $X_i^*$.  
Then, $\psi'$ induces a local homeomorphism $\psi'' \colon e(W_1) \to
e(W_2)$ such that $\psi'' \circ e = e \circ \psi'$. Notice that each
$e(W_i)$ contains the intersection of $X_i$ with a cone in $(\C^3,0)$
around the exceptional line represented by $p_i$.
   
 Consider a pair of points  $q=(x,xv,xw)$ and $q'=(x',x'v',x'w')$ in $e(W_1)$. By definition of $\psi''$, we have    
$$||q-q'|| =  || h_1(v,w)(1,v,w) - h_1(v',w')(1,v',w') ||_{\C^3}, $$
$$||\psi''(q)-\psi''(q')|| =||  h_2(\psi(v,w)) (1,\psi(v,w)) -  h_2(\psi(v',w')) (1,\psi(v',w')) ||_{\C^3}.$$
Then   $(\ast)$ implies that the 
ratio $\frac{||\psi''(q)-\psi''(q')||}{||q-q'||}$ is bounded above and below in a neighborhood of the origin. 

Now let  $\widetilde W_i$ be the union of the
  $W_i$'s
and let  $\psi' \colon \widetilde W_1 \to
  \widetilde W_2$ be the homeomorphism whose restriction to each
  $W_1$ is the local $\psi'$. Then $\psi''\colon e(\widetilde W_1)\to e(\widetilde W_2)$ is the outer bilipschitz homeomorphism induced by
  $\psi'$ and we must extend $\psi''$ over all of $X_1$.

  Let $B$ be a Milnor ball for $X_1$ and $X_2$ around $0$. We set
  $\widetilde{Y_i} = \overline{(e^{-1}(B \cap X_i)\setminus
    \widetilde{W_i}}$.  For $i=1,2$ we can adjust $\widetilde{W_i}$ so
  that $\widetilde{Y_i}$ is a $D^2$-bundle over the exceptional
  divisor $C_i$ minus its intersection with $\widetilde{W_i}$, i.e.,
  over $\widetilde C_i:=\overline{C_i \setminus \widetilde{W_i}}$, and
  whose fibers are curvettes of $C_i$. We want to extend $\psi''\colon e(\widetilde W_1)\to e(\widetilde W_2)$ to a bilipschitz map over
  the conical regions $e(\widetilde Y_1)$ and $e(\widetilde Y_2)$. For
  this it suffices to extend $\psi'$ by a bundle isomorphism 
  $\widetilde Y_1 \to \widetilde Y_2$, since the resulting
  $e(\widetilde Y_1) \to e(\widetilde Y_2)$ is then bilipschitz.

  $(X_1,0)$ and $(X_2,0)$ are inner bilipschitz equivalent by Theorem
  \ref{thm:inner}), so by \cite[1.9 (2)]{BNP} the image by $\psi'' $
  of the foliation of $e(\widetilde W_1)$ by Milnor fibers of a
  generic linear form $\ell_1$ has the homotopy class of the
  corresponding foliation by fibers of $\ell_2$ in $e(\widetilde
  W_2)$. Since the projectivized tangent cones $C_1$ and $C_2$ are
  reduced, a fiber of $\ell_i\circ e$ intersects each $D^2$-fiber over
  $\partial \widetilde C_i$ in one point. This gives a trivialization
  of the $D^2$-bundle over each $\partial \widetilde C_i$ and
  therefore determines a relative Chern class for each component of
  the bundle $\widetilde Y_i$ over $\widetilde C_i$. The map $\psi'$
  restricted to the bundle over $\partial \widetilde C_1$ extends to
  bundle isomorphisms between the components of $\widetilde Y_1$ and
  $\widetilde Y_2$ if and only if their relative Chern classes
  agree. But for $i=1,2$ these relative Chern classes are given by the
  negative of the number of intersection points of $\ell_i^*$ with
  each component of $C_i$ (i.e., the degrees of these components of
  $C_i$), and these degrees agree since $C_1$ and $C_2$ are
  combinatorially equivalent.

We have now constructed a map $\psi''\colon
  (X_1,0)\to (X_2,0)$ which is outer bilipschitz if we restrict to distance
  between pairs of points $x,y$ which are either both in a single
  component of $e(\widetilde W_1)$ or both in the conical region
  $e(\widetilde Y_1)$. Let $N\widetilde Y_i$ be a larger version of
  the bundle $\widetilde Y_i$, so $e(N\widetilde Y_i)$ is a conical
  neighborhood of $e(\widetilde Y_i)$. We still have an outer
  bilipschitz constant for $\psi''$ for any $x$ and $y$ which are both in a
  single component of $e(\widetilde W_1)$ or both in the conical region
  $e(N\widetilde Y_1)$. Otherwise, either one of $x,y$ is in $e(\widetilde
  W_1)\setminus e(N\widetilde Y_1)$ and the other in $e(\widetilde Y_1)$
  or $x$ and $y$ are in different components of $e(\widetilde W_1)$. The ratio of inner to outer distance is clearly bounded for such point pairs, so since $\psi''$ is inner bilipschitz, it is outer bilipschitz.
\end{proof}

\end{document}